\documentclass[11pt]{article}
 
\usepackage{graphicx}
\usepackage{color}
\usepackage{amsfonts}
\usepackage{amsmath}
\usepackage[mathcal]{eucal}
\usepackage{amsthm}
\usepackage{amssymb}
\usepackage{caption}
\usepackage{enumerate}
\usepackage{a4wide}
\usepackage{epstopdf}

\newcommand{\qitem}[1]{\noindent\leavevmode\hangindent1.5\parindent%
  \noindent\hbox to1.5\parindent{#1\hss}\ignorespaces}
\newcommand{\edp}[2]{#1^{[\natural #2]}}

\def\soft#1{\leavevmode\setbox0=\hbox{h}\dimen7=\ht0\advance
    \dimen7 by-1ex\relax\if t#1\relax\rlap{\raise.6\dimen7
    \hbox{\kern.3ex\char'47}}#1\relax\else\if T#1\relax
    \rlap{\raise.5\dimen7\hbox{\kern1.3ex\char'47}}#1\relax
    \else\if d#1\relax\rlap{\raise.5\dimen7\hbox{\kern.9ex
    \char'47}}#1\relax\else\if D#1\relax\rlap{\raise.5\dimen7
    \hbox{\kern1.4ex\char'47}}#1\relax\else\if l#1\relax
    \rlap{\raise.5\dimen7\hbox{\kern.4ex\char'47}}#1\relax
    \else\if L#1\relax\rlap{\raise.5\dimen7\hbox{\kern.7ex
    \char'47}}#1\relax\else\message{accent \string\soft
    \space #1 not defined!}#1\relax\fi\fi\fi\fi\fi\fi}

\newtheorem{theo}{Theorem}[section]
\newtheorem{mydef}[theo]{Definition}
\newtheorem{prop}[theo]{Proposition}
\newtheorem{coro}[theo]{Corollary}
\newtheorem{lem}[theo]{Lemma}

\newtheorem{prob}[theo]{Problem}

\usepackage{soul}

\title{Colouring exact distance graphs of chordal graphs}

\author{Daniel A. Quiroz \thanks{\,\textit{Current affiliation:} Departamento de Ingenier\'ia Matem\'atica and Centro de Modelamiento Matem\'atico, Universidad de Chile, Santiago, Chile. Email:
    \texttt{dquiroz@dim.uchile.cl}.} \\ Department of Mathematics \\ London
    School of Economics and Political Science\\ London, WC2A 2AE, UK}
\date{}
\begin{document} 
\maketitle
\vskip .15 in
\begin{abstract} 
  For a graph $G=(V,E)$ and positive integer~$p$, the \emph{exact
    distance-$p$ graph}~$\edp{G}{p}$ is the graph with vertex set~$V$ and
  with an edge between vertices~$x$ and~$y$ if and only if~$x$
  and~$y$ have distance~$p$. Recently, there has been an effort to obtain
  bounds on the chromatic number $\chi(\edp{G}{p})$ of exact distance-$p$
  graphs for $G$ from certain classes of graphs. In particular, if a
  graph $G$ has tree-width $t$, it has been shown that
  $\chi(\edp{G}{p}) \in \mathcal{O}(p^{t-1})$ for odd $p$, and
  $\chi(\edp{G}{p}) \in \mathcal{O}(p^{t}\Delta(G))$ for even $p$. We
  show that if $G$ is chordal and has tree-width $t$, then
  $\chi(\edp{G}{p}) \in \mathcal{O}(p\, t^2)$ for odd $p$, and $\chi(\edp{G}{p})
  \in \mathcal{O}(p\, t^2 \Delta(G))$ for even~$p$.

If we could show that for every graph $H$ of tree-width $t$ there is a chordal graph~$G$ of tree-width $t$ which contains $H$ as an \emph{isometric subgraph} (i.e., a distance preserving subgraph), then our results would extend to all graphs of tree-width $t$. While we cannot do this, we show that for every graph $H$ of genus $g$  there is a graph $G$ which is a triangulation of genus $g$ and contains $H$ as an isometric subgraph.

\end{abstract}
\vskip .1 in

\section{Introduction}\label{sec1}

All graphs in this paper are assumed to be finite, undirected, simple and
without loops. For a graph $G=(V,E)$ and vertices $u,v \in V$, we
denote by $d_G(u,v)$ (or $d(u,v)$ when there is no danger of ambiguity) the
distance between $u$ and $v$, that is, the number of edges in a shortest
path between $u$ and $v$.

For a positive integer~$p$, the \emph{$p$-th power} graph $G^p=(V,E^p)$
of $G$ has the same vertex set as $G$, and the pair $uv$ belongs to
$E^p$ if and only if $d_G(u,v)\le p$.

Problems related to the chromatic number of graph powers were first
considered by Kramer and Kramer~\cite{KK69a,KK69b} in 1969 and have enjoyed
significant attention ever since. It is clear that for $p\ge 2$ any power
of a star is a clique. Hence, in order to obtain bounds on $\chi(G^p)$ we
need to use the maximum degree $\Delta(G)$ of $G$. One can easily see that
any graph $G$ with $\Delta(G) \ge 3$ satisfies
\[ \chi(G^p)\le 1+\Delta(G^p)\le
1+\Delta(G)\cdot\sum_{i=0}^{p-1}(\Delta(G)-1)^i\in
\mathcal{O}(\Delta(G)^p).\]

However, there are many classes of graphs for which much better bounds can
be obtained. Recall that a graph is \emph{$k$-degenerate} if every subgraph
of $G$ contains a vertex of degree at most~$k$. Parametrising in terms of
the degeneracy, Agnarsson and Halld\'orsson~\cite{AH03} gave upper bounds
for many classes of graphs.

\begin{theo}[Agnarsson and
  Halld\'orsson~\cite{AH03}]\label{largepow}\mbox{}\\*
  Let $k$ and $p$ be positive integers. There exists $c=c(k,p)$
  such that for every $k$-degenerate graph $G$ we have $\chi(G^p)\le
  c\cdot\Delta(G)^{\lfloor p/2\rfloor}$.
\end{theo}

\noindent
Note that the exponent on $\Delta(G)$ in this result is best possible, even
for the class of trees, as the complete $\Delta$-ary tree with radius
$\lfloor p/2 \rfloor$ attests~\cite{AH03}.

For some classes of graphs it is possible to obtain similar bounds without
parametrising in terms of the degeneracy. Recall that a graph $G$ is
chordal if every cycle of $G$ has a chord, i.e., if every induced cycle is
a triangle. In \cite{K04}, Kr\'{a}\soft{l} proved that every chordal graph $G$
with maximum degree $\Delta$ satisfies $\chi(G^p)\in
\mathcal{O}(\sqrt{p}\Delta^{(p+1)/2})$ for even $p$, and $\chi(G^p)\in
\mathcal{O}(\Delta^{(p+1)/2})$ for odd $p$. Kr\'{a}\soft{l} also showed that this
upper bound for odd $p$ is tight. It is worth mentioning that, in order to
obtain this tight upper bound, Kr\'{a}\soft{l} gave a simple proof of the already known fact that odd powers
of chordal graphs are also chordal \cite{BP83,D78}.

Given that graphs with tree-width at most $t$ have
degeneracy at most $t$, Theorem~\ref{largepow} gives us an upper bound on
$\chi(G^p)$ when $G$ belongs to a graph class with bounded tree-width.
Although tree-width is usually defined in terms of tree-decompositions, an
equivalent definition can be given in terms of chordal graphs, as follows.

\begin{mydef}\label{deftw}
  The \emph{tree-width} $\mathrm{tw}(G)$ of a graph $G$ is the
  smallest integer $t$ such that $G$ is a subgraph of a chordal graph with
  clique number $t+1$.
\end{mydef}

\noindent
A notion related to graph powers is that of exact distance graphs. For a
positive integer~$p$, the \emph{exact distance-$p$ graph}
$G^{[\natural p]}=(V,E^{[\natural p]})$ of $G$ has the same vertex
set as $G$, and the pair~$uv$ belongs to $E^{[\natural p]}$ if and only if
$d_G(u,v)= p$. Clearly, $E^{[\natural p]}$ is a subset of the edge set
of~$G^p$, which means that $\chi(G^{[\natural p]}) \le \chi(G^p)$. We
immediately see that Theorem~\ref{largepow} also gives an upper bound for
the chromatic number of $G^{[\natural p]}$ in terms of the degeneracy and
the maximum degree of $G$. However, when considering exact distance graphs,
this upper bound is far from best possible. This is attested, for instance,
by the following recent result of Van~den~Heuvel, Kierstead and
  Quiroz~\cite{vdHetal}.

\begin{theo}[Van den Heuvel \emph{et
  al.}~\cite{vdHetal}]\label{tw}\mbox{}

  \qitem{(a)} Let $p$ be an odd integer. For every graph $G$ with
  tree-width at most $t$ we have \\ $\chi(G^{[\natural p]}) \le t\cdot
  \binom{p+t-1}{t} +1 \in \mathcal{O}(p^{t-1})$.

  \qitem{(b)} Let $p$ be an even integer. For every graph $G$ with
  tree-width at most $t$ we have \\$\chi(G^{[\natural p]}) \le \Bigl(
  t\cdot \binom{p+t}{t} +1\Bigr)\cdot \Delta(G) \in \mathcal{O}(p^t\cdot
  \Delta(G))$.
\end{theo}

\noindent
In the following sense, this result actually extends to all classes excluding a fixed minor. Let~$\mathcal{K}$ be a graph class excluding a fixed minor. It was first shown in~\cite{NOdM12} that for every odd integer~$p$, there exist a constant $N=N(\mathcal{K},p)$ such that for every graph $G\in\mathcal{K}$ we have $\chi(G^{[\natural p]})\le N$. In~\cite{vdHetal} it is shown that for every even $p$ there exists a
  constant $N'=N'(\mathcal{K},p)$ such that for every graph
  $G\in\mathcal{K}$ we have $\chi(G^{[\natural p]})\le N'\cdot\Delta(G)$. Moreover, these results remain true if we take $\mathcal{K}$ to be any class with bounded expansion.

Our main result is a significant improvement on the bounds of
Theorem~\ref{tw} for chordal graphs. 

\begin{theo}\label{main1}\mbox{}\\*
  Let~$G$ be a chordal graph with clique number $t\ge 2$.

  \smallskip
  \qitem{(a)} For every odd integer $p\ge 1$ we have
  $\chi(G^{[\natural p]}) \le \binom{t}{2}\cdot (p+1)$.

  \smallskip
  \qitem{(b)} For every even integer $p \ge 2$ we have
  $\chi(G^{[\natural p]}) \le \binom{t}{2}\cdot \Delta(G)\cdot (p+1)$.
\end{theo}

\noindent
This result is implied by the more general Theorem~\ref{main2} below.

Although Definition~\ref{deftw} tells us that every graph of tree-width $t$
is a subgraph of a chordal graph with clique number $t+1$,
Theorem~\ref{main1} does not extend to all graphs with tree-width at
most $t$. We shall say more about this at the end of this section. Before
that, let us state the full generality of our results.

For two (labelled) graphs $G=(V,E)$ and $G'=(V,E')$ on the same vertex set, define
$G\cup G'=(V,E\cup E')$. For a fixed positive integer $p$,
Theorem~\ref{main1} trivially gives
$\chi(\edp{G}{p_1}\cup \edp{G}{p_2}\cup \dots \cup\edp{G}{p_s})\le
\binom{t}{2}^{s} \cdot \Delta(G)^q \cdot (p+1)^s$ for any subset
$\{p_1,p_2, \dots, p_s\}$ of $\{1,2, \dots, p\}$ with~$q$ even
  elements. Notice that if we take
$\{p_1,p_2, \dots, p_s\}=\{1,2, \dots, p\}$, then we have
$\edp{G}{p_1}\cup \edp{G}{p_2}\cup \dots \cup\edp{G}{p_s}=G^p$, meaning that Theorem~\ref{main1} implies a version of Theorem~\ref{largepow} for chordal graphs. Taking a
subset of even integers turns out to be quite different from taking a
subset of odd integers. For even $p$, we note that the complete $\Delta$-ary
tree of radius $\lfloor p/2 \rfloor$,~$T_{\Delta,p}$, shows that
$\chi(\edp{T_{\Delta,p}}{2}\cup \edp{T_{\Delta,p}}{4}\cup \dots
\cup\edp{T_{\Delta,p}}{p}) \in \Omega(\Delta^{p/2})$. Hence, the bound of
Theorem~\ref{largepow} gives again the right exponent on $\Delta(G)$. In
contrast, we see that for odd $p$ we obtain an upper bound on
$\chi(\edp{G}{1}\cup \edp{G}{3}\cup \dots \cup\edp{G}{p})$ which does not
depend on~$\Delta(G)$. However, these trivial upper bounds stop being
linear in $p$, even if we simply consider
$\chi(\edp{G}{(p-2)}\cup \edp{G}{p})$.

We prove Theorem~\ref{main1} by proving the following stronger result which
gives upper bounds on the chromatic number of all these gradations
between~$G^{[\natural p]}$ and $G^p$. For instance, these upper bounds are
linear in~$p$ if the size of the subsets of $\{1,2, \dots p\}$ considered does not grow with~$p$. 

\begin{theo}\label{main2}\mbox{}\\*
  Let~$G$ be a chordal graph with clique number $t\ge 2$. Let $p$ be a
  positive integer, $S=\{p_1,p_2, \dots ,p_s\}\subseteq \{1,2,\dots ,p\}$
  and $q$ be the number of even integers in $S$.

  \smallskip
  \qitem{(a)} If $1\notin S$, then we have
  $\chi(\edp{G}{p_1}\cup \edp{G}{p_2}\cup \dots \cup\edp{G}{p_s}) \le
  \binom{t}{2}^{s} \cdot \Delta(G)^q \cdot (p+1)$.

  \smallskip
  \qitem{(b)} If $1\in S$, then we have
  $\chi(\edp{G}{p_1}\cup \edp{G}{p_2}\cup \dots \cup\edp{G}{p_s}) \le
  t\cdot \binom{t}{2}^{s-1} \cdot \Delta(G)^q\cdot (p+1)$.
\end{theo}

\noindent
Of course, if $S=\{1\}$ then we have that
$\chi(\edp{G}{p_1}\cup \edp{G}{p_2}\cup \dots \cup\edp{G}{p_s})=\chi(G)=t$,
given the well known fact that chordal graphs are perfect and hence satisfy
$\chi(G)=\omega(G)$.

We obtain Theorem~\ref{main2} by partitioning the graph $G$ into levels. We
fix a vertex $x\in V(G)$ and we define the level $\ell$ as the set of
vertices having distance $\ell$ with $x$. We bound the number of colours
needed to colour one level and then give different colours to levels
which are at distance at most $p$. Apart from being natural in the context
of exact distance graphs, this simple levelling argument is regularly used
in colouring problems related to perfect graphs. (The
  real problem is, of course, in the analysis of each level.)
K\"{u}ndgen and Pelsmajer~\cite{KP08} used level partitions of chordal
graphs to find an upper bound on the number of colours needed in a
nonrepetitive colouring of a graph with tree-width $t$. Level partitions also play a key role in a series of papers of Chudnovsky,
Scott, Seymour and Spirkl which starts off by proving a conjecture
of Gy\'{a}rf\'{a}s stating that there is a function $f$ such that
$\chi(G)\le f(\omega(G))$ for every graph $G$ with no odd hole~\cite{SS16}.

Together with level partitions, the notion of an adjacent-cliques graph is
key in our proof of Theorem~\ref{main2}. For a graph $G$ and two
cliques $K$ and~$K^*$ in $G$, we say that $K$ and~$K^*$ are
\emph{adjacent} if they are disjoint and there are vertices
$x\in K$, $y\in K^*$ with $xy \in E(G)$. The \emph{adjacent-cliques graph}
$\mathit{AC}(G)$ of a graph $G$ has a vertex for each clique of~$G$, and
two vertices~$K$ and $K^*$ of $\mathit{AC}(G)$ are adjacent if and only if
their corresponding cliques in~$G$ are adjacent. We prove the following
result for the chromatic number of $\mathit{AC}(G)$ when~$G$ is chordal.

\begin{theo}\label{adj}\mbox{}\\*
  Let $G$ be a chordal graph with clique number at most $t$. We have
  $\chi(\mathit{AC}(G)) \le \binom{t+1}{2}$.
\end{theo}

\noindent
We denote the line graph of a graph $G$ by $L(G)$. It is easy
to see that $\mathit{AC}(G)$ contains~$G$ and $L(G)^{[\natural 2]}$ as
subgraphs. Hence, Theorem~\ref{adj} tells us that for all chordal
graphs~$G$ with clique number $t$ there is a constant $c(t)$ such that
$\chi(L(G)^{[\natural 2]})\le c(t)$. Contrast this with the fact that even for
trees, $L(G)^{[\natural p]}$ can have arbitrarily large cliques if $p$ is
odd (consider stars and subdivided stars). Whether or not there are similar constant upper
bounds on $\chi(L(G)^{[\natural p]})$ for all even $p\ge4$, we leave as an open problem. Also, in light of known results about the strong chromatic index~\cite{DGS15}, we conjecture that for every $k$ there is a constant $c(k)$, such that $\chi(L(G)^{[\natural 2]})\le c(k)$ for every $k$ degenerate graph $G$. 

Let $H$ be a subgraph of a graph $G$. We say $H$ is an \emph{isometric subgraph} of $G$ if $d_H(u,v)=d_G(u,v)$ for every $u,v\in V(H)$. Note that if $H$ is an isometric subgraph of $G$, then $\chi(H^{[\natural p]}) \le \chi(G^{[\natural p]})$ for every positive
  integer~$p$. Thus, if we could show that for every graph $H$ of tree-width $t$ there is a chordal graph~$G$ of tree-width $t$ which contains $H$ as an isometric subgraph, then we could extend Theorems~\ref{main1} and~\ref{main2} to all graphs of tree-width $t$. While we cannot do this we prove the following.

\begin{prop}\label{genus}\mbox{}\\*
  Let $H$ be a graph with genus $g \ge 0$. There is a graph $G$ which is a triangulation of genus $g$ and contains $H$ as an isometric subgraph.
\end{prop}

\noindent
While there are results showing, for instance, that every graph is an isometric subgraph of a vertex-transitive graph~\cite{DS08}, Proposition~\ref{genus} is new as far as we are aware.

By Proposition~\ref{genus} we have, for instance, that an upper bound on $\chi(G^{[\natural p]})$ for all maximal planar graphs implies an upper bound on $\chi(G^{[\natural p]})$ for all planar graphs. In~\cite{vdHetal}, Van den Heuvel \emph{et al.} prove that $\chi(G^{[\natural 3]})\le 105$ for all planar graphs $G$. While this represents a major improvement on previous upper bounds (a bound of $5 \cdot 2^{20,971,522}$ is implied by~\cite{NOdM12}, and one of $70\cdot 2^{70}$ by~\cite{S16}), it is still far from the lower bound of 7, given also in~\cite{vdHetal}. Proposition~\ref{genus} opens the gates for classical techniques for colouring triangulations of planar graphs to be applied on this problem.

The rest of the paper is organised as follows. In the next section we study
the properties of level partitions of chordal graphs which will be
essential for the proof of Theorem~\ref{main2}. In Section~\ref{sec3} we
prove Theorem~\ref{adj}, and in Section~\ref{sec4} we complete the proof of
Theorem~\ref{main2}. In Section~\ref{sec5} we prove
Proposition~\ref{genus}. We conclude with a short discussion on lower bounds and by mentioning some open problems.

\section{Level partitions of chordal graphs}\label{sec2}

Let $G$ be a graph and $x$ be a fixed vertex of $G$. For any positive
integer $\ell$, set $N^\ell(x)=\{ v \in V(G) \mid d(v,x)= \ell \}$. We call
$N^\ell(x)$ the \emph{$\ell$-th level} of~$G$ with respect to~$x$, and if
we set $N^0(x)=\{ x\}$ we get that these levels partition the vertices of the connected
component of $G$ containing $x$. We also set
$N^{<\ell}(x)=\bigcup_{i<\ell}N^i(x)$ and
$N^{>\ell}(x)=\bigcup_{i>\ell}N^i(x)$.

Let $G_\ell$, $G_{<\ell}$ and $G_{>\ell}$ be the graphs induced by $N^\ell(x)$,
$N^{<\ell}(x)$ and $N^{>\ell}(x)$, respectively. Define the
\emph{$\ell$-shadow} of a subgraph~$H$ of~$G$ as the set of
vertices in~$N^\ell(x)$ which have a neighbour in~$V(H)$. We say
that~$G$ is \emph{shadow complete} (with respect to~$x$) if for every
non-negative integer $\ell$, the $\ell$-shadow of every connected component
of $G_{>\ell}$ induces a complete graph.

Using a well-known theorem of Dirac~\cite{D61} which characterises
chordal graphs in terms of their minimal vertex cut sets, K\"{u}ndgen and
Pelsmajer~\cite{KP08} proved that connected chordal graphs are shadow
complete with respect to any vertex.

\begin{lem}[K\"{u}ndgen and Pelsmajer~\cite{KP08}]\label{shadow}\mbox{}\\*
  Let $G$ be a connected chordal graph with clique number $t\ge 2$ and let
  $x$ be any vertex in $V(G)$. Then $G$ is shadow complete with respect to
  $x$ and every $G_\ell$ is a chordal graph with clique number strictly
  smaller than $t$.
\end{lem}

\noindent
Before we start to see some implications of this lemma, let us state one
additional definition. We say that a vertex $v \in N^\ell(x)$, is an
\emph{ancestor} (with respect to~$x$) of a vertex $u \in N^m(x)$,
$\ell<m$, if there is a path between $u$ and $v$ of length $m-\ell$.
Clearly, any such path has exactly one vertex in each level $N^\ell(x),
N^{\ell+1}(x), \dots , N^m(x)$.

The following result follows directly from Lemma~\ref{shadow}.

\begin{coro}\label{desce}\mbox{}\\*
  Let $G$ be a connected chordal graph, $x \in
  V(G)$, and $u, v \in N^\ell(x)$ for some positive integer~$\ell$. If $u$
  and $v$ are both ancestors of some $y \in N^{>\ell}(x)$, then $u$ and $v$
  are neighbours.
\end{coro}

\noindent
With a bit more care we can prove that if two vertices are at the same
level $\ell$ and are at distance $p$, then their ancestors at level
$\ell-\lfloor p/2 \rfloor$ form cliques which either intersect or are
adjacent.

\begin{lem}\label{path}\mbox{}\\*
  Let $G$ be a connected chordal graph,
  $x \in V(G)$, and $u, v \in N^\ell(x)$ for some positive integer~$\ell$.
  Suppose $d(u,v)=p \ge 2$ and let $K_u, K_v$ be the (complete) graphs induced by the ancestors of
  $u$ and $v$ in $N^{\ell-\lfloor p/2 \rfloor}(x)$, respectively. We have
  that

  \smallskip
  \qitem{(a)} if $p$ is odd, then $K_u$ and $K_v$ are adjacent;

  \smallskip
  \qitem{(b)} if $p$ is even, then $K_u$ and $K_v$ are adjacent or
  $K_u \cap K_v \ne \varnothing$.
\end{lem}

\begin{proof}
  Let $k=\lfloor p/2 \rfloor$, and note that we must have $\ell \ge k$ as
  otherwise there would be a walk from $u$ to $v$ that goes through $x$ and
  has length $2l<2k\le p$, which would contradict $d(u,v)=p$. We prove
  (a) and (b) simultaneously by considering two possibilities for $u$ and
  $v$.

  We first consider the case in which~$u$ and~$v$ are in different
  components of~$G_{>\ell-k}$. In this case it is clear that every path of
  length $p$ joining $u$ and $v$ must contain a vertex from~$G_{\ell-k}$
  (and no vertices in $G_{<\ell-k}$). It is also easy to see that if $p$ is
  even, then every path of length~$p$ joining~$u$ and $v$ must have exactly
  one vertex in~$G_{\ell-k}$. Since $d(u,v)=p$, this means that $K_u \cap
  K_v \ne \varnothing$. If $p$ is odd, then every path of length $p$
  joining~$u$ and $v$ must have exactly two vertices in~$G_{\ell-k}$. This
  implies that $K_u$ and $K_v$ are adjacent.

  We are now left to consider the case in which $u$ and $v$ are in the same
  connected component~$C$ of~$G_{>\ell-k}$. Let $z\in G_{\ell-k+1}$ be an
  ancestor of $u$ and let $z' \in G_{\ell-k+1}$ be an ancestor of~$v$.
  Clearly, $z$ and $z'$ belong to $C$. We know by Lemma~\ref{shadow}
  that since $G$ is chordal it is shadow complete, and so the neighbours of
  $z$ and $z'$ in $G_{\ell-k}$ form a clique. This means that either $K_u$
  and $K_v$ are adjacent or $K_u \cap K_v \ne \varnothing$. However, if $p$
  is odd we cannot have $K_u \cap K_v \ne \varnothing$.
\end{proof}

\section{Adjacent-cliques graphs}\label{sec3}

In this section we prove Theorem~\ref{adj}. In order to prove this result
we need to recall a specific characterisation of chordal graphs.

A \emph{perfect elimination ordering} of a graph $G$ is a linear
ordering~$L$ of $V(G)$ such that, for every vertex $v \in V(G)$, the
neighbours of $v$ which are smaller than $v$ in $L$ form a clique. The
following classical result is proved in \cite[Section~7]{FG65}.

\begin{prop}[Fulkerson and Gross~\cite{FG65}]\label{peo}\mbox{}\\*
  A graph is chordal if and only if it has a perfect elimination ordering.
\end{prop}

\begin{proof}[Proof of Theorem~\ref{adj}] 
  By Proposition~\ref{peo} we know $G$ has a perfect elimination ordering.
  We fix one such ordering $L$. We say a vertex $u$ is a
  \emph{predecessor} of a vertex $v$ if $uv \in E(G)$ and $u <_L v$.
  Moving along the ordering $L$, we colour the vertices of $G$ in
  the following way. A vertex~$v$ gets a colour~$a(v)$ which is different
  from $a(u)$ if $u$ is a predecessor of $v$ or $u$ is a predecessor
  of a predecessor of $v$. Since the clique number of $G$ is at most $t$
  and since~$L$ is a perfect elimination ordering, each vertex has at most
  $t-1$ predecessors. Moreover, by choice of~$L$ we have that if $v$ has
  $r\le t-1$ predecessors, the largest (with respect to $L$) of its
  predecessors has at most $t-r$ predecessors which are not already
  predecessors of $v$; the second largest predecessor of~$v$ has at most
  $t-(r-1)$ predecessors which are not already predecessors of $v$, and so
  on. Therefore, for any vertex $v$ the set of predecessors and
  predecessors of a predecessor of $v$ has size at most $r+(t-r)+(t-(r-1))+
  \dots +(t-1)\le (t-1)+1+2+ \dots + t-1=\binom{t+1}{2}-1$. And so, the
  colouring $a$ uses at most $\binom{t+1}{2}$ colours.

  We define a colouring $c$ on the vertices of $\mathit{AC}(G)$ in the
  following way. For every vertex~$K$ in $\mathit{AC}(G)$, (i.e., for every clique $K$ in $G$) we set $\mu (K)$ as the smallest vertex
  of $K$ with respect to $L$. Every vertex $K$ is assigned the colour
  $c(K)=a(\mu (K))$. To prove the theorem, it suffices to show that $c$ is a proper colouring of
  $\mathit{AC}(G)$.

  Let $K$ and $K^*$ be adjacent vertices in $\mathit{AC}(G)$. We must show that we have
  $a(\mu(K))\ne a(\mu(K^*))$. Let $u,u' \in K$ and $v,v' \in K^*$ be
  vertices of~$G$ such that $u=\mu(K)$, $v=\mu (K^*)$ and $u'v'\in E(G)$.
  Without loss of generality we assume that $u' <_L v'$. If $v=v'$, we have
  that $u'v\in E(G)$. Otherwise, we have that both $u'$ and $v$ are
  predecessors of $v'$. Since $L$ is a perfect elimination ordering, we
  also obtain $u'v \in E(G)$.

  Since $a$ is a proper colouring of $G$, if we have $u=u'$ we
  immediately get that $a(\mu (K))=a(u) \ne a(v) =a(\mu (K^*))$ as desired.
  So assume $u \ne u'$. If $v<_L u'$, we have that
  both $u$ and~$v$ are predecessors of $u'$, and so $uv \in E(G)$. This
  again gives us that $a(\mu (K)) \ne a(\mu (K^*))$. Otherwise, if $u'<_L
  v$ we have that $u'$ is a predecessor of~$v$. Since $u$ is a predecessor
  of~$u'$ we also obtain that $a(\mu (K))=a(u) \ne a(v) =a(\mu (K^*))$ by
  definition of~$a$.
\end{proof}

\noindent
For later use, we note a property of the colouring
$c$ we defined in the previous proof.

\begin{lem}\label{AIC}\mbox{}\\*
  Let $G$ be a chordal graph, and $K, K^*$ clique subgraphs of $G$. Let $c$ and $\mu$ be as defined in the proof of
  Theorem~\ref{adj}.  If  $K\cap K^*\ne \varnothing$ and $c(K)=c(K^*)$, then we have $\mu(K)=\mu(K^*)$.
\end{lem}

\begin{proof}
  We prove that if we have $K\cap K^*\ne \varnothing$ and $\mu(K) \ne \mu(K^*)$, then we have $c(K) \ne c(K^*)$. We do this by proving  that
  $\mu(K)$ and $\mu (K^*)$ are adjacent in $G$. Since $a$ is a proper
  colouring of $G$, this will tell us that
  $c(K)=a(\mu (K)) \ne a(\mu (K^*))=c(K^*)$ which gives us the result.

  If $\mu(K)$ and $\mu (K^*)$ are not adjacent, we have that neither
  of $\mu(K),\mu (K^*)$ belong to $K \cap K^*$. Therefore, the
  minimum vertex $v$ in $K \cap K^*$ with respect to $L$ is adjacent
  to~$\mu(K)$ and $\mu (K^*)$, and $\mu(K),\mu (K^*)$ are smaller
  than $v$ in $L$. This contradicts the choice of~$L$, since 
 the neighbours of $v$  smaller than $v$ in $L$ must form a clique,
  and thus be pairwise adjacent.
\end{proof}

\section{Exact distance graphs of chordal graphs}\label{sec4}

Theorem~\ref{main2} will follow from the next lemma.

\begin{lem}\label{level}\mbox{}\\*
  Let~$G$ be a connected chordal graph with clique number $t\ge 2$, let~$x$
  be a vertex in~$G$, and $p\ge2$ an integer. For any non-negative
  integer~$\ell$, we have that

  \smallskip
  \qitem{(a)} if $p$ is odd, then there is a colouring $h$ of
  $N^{\ell}(x)$ using at most $\binom{t}{2}$ colours such that if
  $u,v \in N^{\ell}(x)$ satisfy $uv \in E(G^{[\natural p]})$, then
  $h(u) \neq h(v)$;

  \smallskip
  \qitem{(b)} if $p$ is even, then there is a colouring $h'$ of
  $N^{\ell}(x)$ using at most $\binom{t}{2}\cdot \Delta(G)$ colours such
  that if $u,v \in N^{\ell}(x)$ satisfy $uv \in E(G^{[\natural p]})$, then
  $h'(u) \neq h'(v)$.
\end{lem}

\begin{proof}
  Let $k= \lfloor p/2 \rfloor$ and note, just as in the proof of
  Lemma~\ref{path}, that if $u,v \in N^{\ell}(x)$ satisfy
  $uv \in E(G^{[\natural p]})$, then we must have $\ell\ge k$. By Lemma~\ref{shadow} we know that $G_{\ell-k}$ is a chordal
  graph with clique number at most $t-1$. By Theorem~\ref{adj} we know that
  there is a proper colouring $c$ of the vertices of
  $\mathit{AC}(G_{\ell-k})$ which uses at most $\binom{t}{2}$ colours.

   For each vertex $y \in N^{\ell}(x)$ we consider the set of vertices
  $K_y \subseteq N^{\ell-k}(x)$ which are ancestors of $y$. By
  Corollary~\ref{desce} we have that $K_y$  is a vertex of
  $\mathit{AC}(G_{\ell-k})$. 

(a)\quad Define the colouring $h$ by assigning
  $h(y)=c(K_y)$ to every vertex $y \in N^{\ell}(x)$. Let
  $u,v \in N^{\ell}(x)$ be such that $uv \in E(G^{[\natural p]})$. By
  Lemma~\ref{path}\,(a) we have $K_uK_v \in E(\mathit{AC}(G_{\ell-k}))$. Therefore, we have
  $h(u)=c(K_u)\ne c(K_v)=h(v)$, as desired.

  (b)\quad For each  $w \in N^{\ell-k}(x)$ we choose
  an injective function $b_w:N(w)\to\{1, \dots, \Delta(G)\}$. For every
  vertex $y \in N^\ell(x)$ we choose an arbitrary vertex $\sigma(y)$
  from $N^{k-1}(y)\cap N(\mu(K_y))$. The colouring~$h'$ assigns
  $h'(y)=(c(K_y),b_{\mu(K_y)}(\sigma(y)))$ to every vertex $y\in
  N^\ell(x)$. Clearly $h'$ uses at most $\binom{t}{2}\cdot \Delta(G)$
  colours.

  Let $u,v \in N^{\ell}(x)$ be such that $uv \in E(G^{[\natural p]})$. We
  must show that $h'(u) \neq h'(v)$. Suppose we have
  $K_u \cap K_v = \varnothing$. By Lemma~\ref{path}\,(b) we know that $K_u$
  and $K_v$ are adjacent. As in part~(a) we obtain $c(K_u)\ne c(K_v)$ and
  so $h'(u)\ne h'(v)$, as desired.

  We can thus assume we have $K_u \cap K_v \ne \varnothing$. We also assume $c(K_u)=c(K_v)$,
  as otherwise we would have $h'(u)\ne h'(v)$. By Lemma~\ref{AIC} we obtain
  that the corresponding cliques $K_u$ and~$K_v$ satisfy
  $\mu(K_u)=\mu(K_v)$. Now notice that $\sigma(u)$ must be different
  from~$\sigma(v)$, as otherwise there would be a walk of length $p-2$
  joining $u$ and $v$ and going through~$\sigma(u)$, which would contradict
  $d(u,v)=p$. Since $b_{\mu(K_u)}$ is injective, we obtain
  $h'(u)\ne h'(v)$, as desired.
\end{proof}

\begin{proof}[Proof of Theorem~\ref{main2}]
  We may assume that $G$ is connected. As we
  mentioned earlier, this theorem follows from Lemma~\ref{level}. Here we
  prove (a) and leave~(b) to the reader.

  Fix a vertex $x \in V(G)$. Define a function $f: V(G) \to \{0, \dots,
  p\}$ which satisfies $f(u)=k$ for all $u \in N^\ell(x)$ with $\ell \equiv
  k$ mod ($p+1$). For each level $N^{\ell}(x)$ and integer $p_i\in \{p_1,
  p_2, \dots ,p_s \}$, we define $g_{\ell,i}$ as the colouring
  of~$N^{\ell}(x)$ guaranteed by Lemma~\ref{level}, which assigns
  different colours to vertices of~$N^{\ell}(x)$ having distance $p_i$. To
  each vertex $u\in N^{\ell}$ we assign a colour $F(u)=(f(u),
  g_{\ell,1}(u), g_{\ell,2}(u) \dots ,g_{\ell,s}(u))$, and we do this for
  all $\ell$. Notice that for every $1\le i \le s$, each vertex $u\in
  N^{\ell}$ can only have distance~$p_i$ with vertices not in
  $N^{<\ell-p}(x)\cup N^{>\ell+p}(x)$. Hence, this colouring guarantees
  that, for all $1\le i \le s$, $u$ gets a different colour from $v$
  whenever $u$ and $v$ have distance $p_i$.
\end{proof}

\section{Proof of Proposition~\ref{genus}}\label{sec5} 

  We may assume $V(H)\ge 3$, as otherwise the result is trivial. We may
  also assume that $H$ is connected.

  Fix an embedding of $H$ in a surface of genus $g$. We first construct
  from $H$ a graph $H'$ of genus $g$ having the property that all of its
  faces have a cycle as its boundary. This can be done without altering
  distances by means of the following operation. Suppose $y \in V(H)$ is a
  cut vertex. There is an ordering $x_1, x_2, \dots ,x_{|N(y)|}$ of the
  vertices in $N(y)$ such that adding an edge between $x_i$ and $x_{i+1}$
  (wherever such an edge does not already exist) would not create
  crossings. Using this ordering we add a path of length 2 between $x_i$
  and $x_{i+1}$ (modulo~$|N(y)|$) if there is no edge joining the pair.
  Clearly $y$ ceases to be a cut vertex after this operation, and no new
  cut vertices are created. We repeat this operation until there are no cut
  vertices. It is easy to see that $H'$ satisfies that all of its faces
  have a cycle as its boundary, and that for every $u,v\in V(H)$ we have
  $d_H(u,v)=d_{H'}(u,v)$.

  If $H'$ is not a triangulation of genus $g$, then there is a face
  of $H$ having as its boundary a cycle $C_k$, with vertices
  $z_0, \dots, z_{k-1}$, for some $k>3$. Inside this face we draw a cycle
  $C_{k-1}$ with edges $e_1, \dots, e_{k-1}$. For all $1 \le i \le k-1$, we
  add edges joining $z_i$ with the endvertices of $e_i$. We also add
  an edge joining $z_0$ with the common endvertex of $e_1$ and
  $e_{k-1}$. It is easy to see that this can be done in such a way that no
  crossings are made, the area between $C_k$ and $C_{k-1}$ is triangulated
  and $C_{k-1}$ is the boundary of a face. Figure~\ref{cycles}
  shows how to do this for $k=7$. We call the resulting embedded
  graph $F$. If $F$ is not a triangulation, we repeat the operation on $F$,
  and we do this until we get a graph $G$ which is a triangulation of genus $g$.

  \begin{figure}[h]
    \centering
    \smallskip
    \captionsetup{justification=centering}
    \includegraphics[height=2.4 in]{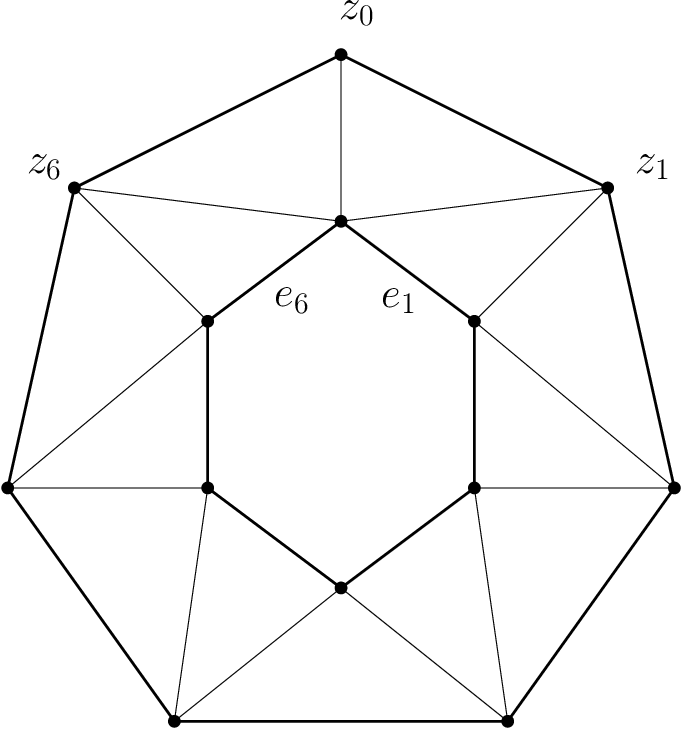}
    \caption{Drawing a $C_6$ inside a face of $H$ which has a $C_7$ as its
      boundary.}
    \label{cycles}
  \end{figure}

  To prove that  we have
  $d_H(u,v)=d_{G}(u,v)$ for every $u,v\in V(H)$, it is enough to prove that we have $d_F(u,v)=d_H(u,v)$ for every
  $u,v \in V(H)$. If this was not the case, there would be a pair of vertices $x,y \in V(C_k)$ such that $F$ contains an $xy$-path of length at
  most $d_H(x,y)-1$ with vertices in~$C_{k-1}$ and no vertices outside of
  $C_k \cup C_{k-1}$. Hence, it suffices to show that every pair of
  vertices $x,y \in V(C_k)$ satisfies $d_{F[C_k \cup C_{k-1}]}(x,y) =
  d_{C_k}(x,y)$. By the construction inside the face having~$C_k$ as its boundary, this is easy to check.

\section{Lower bounds and open problems}\label{sec6}

In~\cite{NOdM15} the following question from Van den Heuvel and
  Naserasr is mentioned: Is there a constant~$C$ such that for every odd
  integer~$p$ and every planar graph~$G$ we have $\chi(G^{[\natural p]})\le
  C$\,? Very recently Bousquet, Esperet, Harutyunyan, and De Joannis de Verclos~\cite{Betal17} gave a
  negative answer to this question in the following way. They constructed a
family of chordal graphs $U_3, U_5, \dots$ with clique
  number~3 which they proved satisfies
$\chi(U_p^{[\natural p]})\in \Omega(\frac{p}{\log (p)})$. This shows that Theorem~\ref{main1}\,(a) is asymptotically best possible (as $p$ tends to infinity), up to a
$\log(p)$ factor.

For any integer $t\ge 1$ and any odd integer $p\ge 1$, it is easy to construct a chordal graph $G$ with clique number $t$ such that $\chi(G^{[\natural p]})=t$. This can be improved, of course. For instance, Van den
Heuvel \emph{et al.}~\cite{vdHetal} constructed a chordal graph with clique number 3 such that its exact distance-3 graph has chromatic number 5. But while the upper bound on $\chi(G^{[\natural p]})$ given by Theorem~\ref{main1}\,(a) is quadratic on $t$, we cannot give any superlinear lower bound on $t$.

A consequence of Theorem~\ref{main2}\,(a) is that for every odd $p$, there
is a constant $N_{p,t}$ such that
$\chi(\edp{G}{1}\cup\edp{G}{3}\cup \dots\cup\edp{G}{p}) \le N_{p,t}$ for
all chordal graphs with clique number $t$. Ne\v{s}et\v{r}il and Ossona de
Mendez~\cite{NOdM15} gave a construction that shows that this
constant must grow with~$p$, even for chordal graphs with clique number 3.
Figure~\ref{largek} gives a simpler construction with the same property.
However, we note that the construction given by Ne\v{s}et\v{r}il and Ossona
de Mendez can be generalised to show that for every $t\ge3$ and odd
positive $p$, there is a chordal graph~$G$ with clique number $t$ such that
$\chi(\edp{G}{1}\cup\edp{G}{3}\cup \dots\cup\edp{G}{p}) \in
\Omega(t^{\lfloor p/2 \rfloor+1})$. Meanwhile, Theorem~\ref{main2}\,(a)
gives that $\chi(\edp{G}{1}\cup\edp{G}{3}\cup \dots\cup\edp{G}{p}) \in
\mathcal{O}(t^{2\lfloor p/2 \rfloor+2})$.

\begin{figure}[h]
 \centering
 \captionsetup{justification=centering}
 \bigskip
 \includegraphics[height=0.23 in]{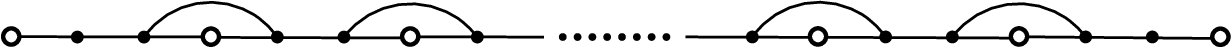}
 \medskip
 \caption{Chordal graphs $G$ with clique number 3 for which $\omega(G^{odd})$, and
   hence $\chi(G^{odd})$, can be arbitrarily large.}
  \label{largek}
\end{figure}

For a graph $G$, a natural generalisation of
$\edp{G}{1}\cup \edp{G}{3}\cup \dots \cup\edp{G}{p}$ is the
graph~$G^{odd}$, which has the same vertex set as~$G$, and~$xy$ is an edge
in~$G^{odd}$ if and only if $x$ and $y$ have odd distance. Both of the
constructions mentioned above tell us that, even for chordal graphs~$G$ with clique number 3,
the chromatic number of~$G^{odd}$ can be arbitrarily large, by witnessing that the
clique number~$\omega(G^{odd})$ can be arbitrarily large. The fact that these constructions are also planar
inspired the following question of Thomass\'{e}, which
appears in~\cite{NOdM12} (also \cite{NOdM15}).

\begin{prob}[{\cite[Problem 11.2]{NOdM12}}]\mbox{}\\*
  Is there a function~$f$ such that for every planar graph~$G$ we have
  $\chi(G^{odd})\le f(\omega(G^{odd}))$\,?
\end{prob}

\noindent
We ask whether there is a function $f_t$ such that for every chordal graph $G$ with clique number~$t$ we have
  $\chi(G^{odd})\le f_t(\omega(G^{odd}))$.

\section*{Acknowledgements}

The author thanks Barnaby Roberts for fruitful discussions at early stages
of this project, and in particular for providing the construction in
Figure~\ref{largek}. The author also thanks Jan van den Heuvel for his
helpful feedback on several drafts, Louis Esperet for sending us
  an early draft of~\cite{Betal17}, and anonymous referees for careful reading and helpful comments.

The author thankfully acknowledges support from CONICYT, PIA/Concurso Apoyo a Centros Cient\'ificos
y Tecnol\'ogicos de Excelencia con Financiamiento Basal AFB170001.

\end{document}